\documentclass[12pt]{amsart}
\usepackage{amssymb}
\usepackage{hyperref}

\theoremstyle{plain}
\newtheorem{theorem}{Theorem}[section]

\newtheorem{corollary}[theorem]{Corollary}
\newtheorem{lemma}[theorem]{Lemma}
\newtheorem{claim}{Claim}

\theoremstyle{definition}

\newtheorem{remark}[theorem]{Remark}

\theoremstyle{remark}
\newtheorem{notation}[theorem]{Notation}

\numberwithin{equation}{section}

\newcommand{\N}{\mathbb N}
\newcommand{\Z}{\mathbb Z}

\newcommand{\R}{\mathbb R}
\newcommand{\C}{\mathbb C}

\newcommand{\SO}{\operatorname{SO}}

\newcommand{\so}{\mathfrak{so}}

\newcommand{\op}{\operatorname}

\newcommand{\spec}{\operatorname{Spec}}

\newcommand{\diag}{\operatorname{diag}}

\newcommand{\ba}{\backslash}
\newcommand{\mult}{\operatorname{mult}}

\newcommand{\norma}[1]{\|{#1}\|_1}

\title[The spectrum on $p$-forms]{The spectrum on $p$-forms of a lens space}
\author{Emilio~A.~Lauret}
\address{Institut f\"ur Mathematik, Humboldt Universit\"at zu Berlin, Unter den Linden 6, 10099 Berlin, Germany.}
\address{Permanent affiliation: CIEM--FaMAF (CONICET), Universidad Nacional de C\'ordoba, Medina Allende s/n, Ciudad Universitaria, 5000 C\'ordoba, Argentina.}
\email{elauret@famaf.unc.edu.ar}
\thanks{This research was partially supported by grants from CONICET, FONCyT and SeCyT--UNC}

\begin{document}

\begin{abstract}
We give an explicit description of the spectrum of the Hodge--Laplace operator on $p$-forms of an arbitrary lens space for any $p$. We write the two generating functions encoding the $p$-spectrum as rational functions. As a consequence, we prove a geometric characterization of lens spaces that are $p$-isospectral for every $p$ in an interval of the form $[0,p_0]$. 
\end{abstract}

\maketitle

\section{Introduction}\label{sec:intro}
Each compact Riemannian manifold $(M,g)$ has an associated \emph{Hodge--Laplace operator} $\Delta_p$ acting on smooth $p$-forms given by $dd^*+d^*d$, where $d$ denotes the exterior derivative and $d^*$ denotes the formal adjoint.
Here $\Delta_0$ is the \emph{Laplace-Beltrami operator}.
The spectrum of $\Delta_p$ is discrete and is usually called the \emph{$p$-spectrum of $(M,g)$},  denoted by $\spec_p(M,g)$ or just $\spec_p(M)$.
Two compact Riemannian manifolds will be said to be \emph{$p$-isospectral} if their $p$-spectra coincide
and in the literature, one abbreviates $0$-isospectral to \emph{isospectral}.
It is well known that $\spec_p(M,g)$ does not determine the geometry of $(M,g)$, as shown by many examples of non-isometric $p$-isospectral manifolds, via the so called \emph{generalized Sunada method} (see for instance \cite{Milnor64}, \cite{Vigneras80}, \cite{Sunada85}, \cite{DeTurckGordon89}) working for all $p$ and also with methods working for individual values of $p$ (see for instance \cite{Gordon86}, \cite{Ikeda88}, \cite{Gornet00}, \cite{MR-p-iso}, \cite{GornetMcGowan06}).

Lens spaces are compact manifolds with positive constant curvature and cyclic fundamental group.
This class of spaces has provided many isospectral examples of different kinds.
Ikeda~\cite{Ikeda80_isosp-lens} gave families of lens spaces mutually $0$-isospectral (see also \cite{Shams11}).
Later in \cite{Ikeda88}, using the same family, he found pairs of lens spaces that are $p$-isospectral for all $0\leq p<p_0$ but not $p_0$-isospectral, for any $p_0>0$.
More recently, Miatello, Rossetti and the author~\cite{LMR-onenorm} found families of pairs, in any odd dimension $n\geq 5$, of lens spaces that are $p$-isospectral for all $p$, but are not \emph{strongly isospectral}.
In particular such pair cannot be constructed by the generalized Sunada method due to DeTurck and Gordon~\cite{DeTurckGordon89}, which uses representation equivalent discrete subgroups.

Furthermore, \cite{LMR-onenorm} also gives an explicit formula for the multiplicity of each eigenvalue in the $0$-spectrum of a lens space and a geometric characterization of $0$-isospectral lens spaces.
The \emph{one-norm} $\norma{\cdot}$ of an element in $\Z^n$ is given by the sum of the absolute values of its entries, and two subsets $\mathcal L$ and $\mathcal L'$ of $\Z^n$ are called $\norma{\cdot}$-isospectral if for each positive integer $k$, there are the same number of elements in $\mathcal L$ and in $\mathcal L'$ with one norm equal to $k$.
Each $(2n-1)$-dimensional lens space $L$ has associated a \emph{congruence lattice} $\mathcal L\subset \Z^n$ (see \eqref{eq1:conglattice}).
Hence, the above mentioned characterization for $0$-isospectral lens spaces can be stated as follows: $L$ and $L'$ are $0$-isospectral if and only if $\mathcal L$ and $\mathcal L'$ are \emph{$\norma{\cdot}$-isospectral}.
(See \cite{BoldtLauret-onenormDirac}, \cite{LMR-survey}, \cite{Lauret-spec0cyclic}, \cite{MohadesHonori17}, \cite{MohadesHonori16} for related results.)

The aim of this paper, as a continuation of the study begun in \cite{LMR-onenorm}, is to give an explicit description of the $p$-spectrum of a lens space for any value of $p$ (Theorem~\ref{thm:eigenvalue-multiplicities}).
As a consequence, for any $p_0\geq0$, we obtain a geometric characterization of lens spaces $p$-isospectral for all $0\leq p\leq p_0$ (Corollary~\ref{cor:charact[0,p]}).
The main tool is a closed explicit formula from \cite{LR-fundstring} for the multiplicity of weights in the irreducible representations of $\SO(2n)$ occurring in the decomposition of $\operatorname{Sym}^k(\C^{2n})\otimes\bigwedge^p (\C^{2n})$ for any $k\geq0$ and $0\leq p\leq n-1$.

Some results of this article have been recently used in \cite{Lauret-computationalstudy} to prove the non-existence of $p$-isospectral lens spaces (and lens orbifolds) for $p$ in certain subsets of $\{0,1,\dots,n-1\}$.

The paper is organized as follows.
Section~\ref{sec:results} introduces the notation and states the results. 
The proofs are included in 
Section~\ref{sec:proofs}.

The author wishes to express his thanks to Roberto Miatello for several helpful comments and his active interest in the publication of this paper.
Furthermore, the author is greatly indebted to the anonymous referee for very carefully reading the manuscript, pointing out many typos and offering valuable suggestions.

\section{Results}\label{sec:results}
We set $G=\SO(2n)$ and $K=\SO(2n-1)$, thus the homogeneous space $G/K$ is diffeomorphic to the $(2n-1)$-dimensional sphere $S^{2n-1}$.
We consider the round metric on $S^{2n-1}$.

If $\Gamma$ is a discrete subgroup of $G$ (thus $\Gamma$ is finite), then $\Gamma\ba S^{2n-1}$ has a structure of a good orbifold, which is a manifold if $\Gamma$ acts freely on $S^{2n-1}$.
In this  case, $\Gamma\ba S^{2n-1}$ is usually called a \emph{spherical space form} and $\Gamma$ is isomorphic to the fundamental group of $\Gamma\ba S^{2n-1}$.

For $\Gamma\subset G$ finite, let us denote by $\Delta_{\Gamma,p}$ the Hodge-Laplace operator on $p$-forms of $\Gamma\ba S^{2n-1}$, which is given by the restriction of $\Delta_p$ on $S^{2n-1}$ to $\Gamma$-invariant smooth $p$-forms on $S^{2n-1}$.
The space $\Gamma\ba S^{2n-1}$ is orientable, thus $\spec_p(\Gamma\ba S^{2n-1}) = \spec_{2n-1-p}(\Gamma\ba S^{2n-1})$, hence $p$-isospectrality for every $0\leq p\leq n-1$ is actually equivalent to $p$-isospectrality for every $p$.

In order to exhibit a description of  $\spec_p(\Gamma\ba S^{2n-1})$, we introduce some notation related to the root system associated to $\mathfrak g_\C:=\so(2n,\C)$. 
We use standard choices for the Cartan subalgebra $\mathfrak h$ and the root system $\Sigma(\mathfrak g_\C,\mathfrak h)$ (see Notation~\ref{notation}).
In particular, the set of positive roots is $\Sigma^+(\mathfrak g_\C,\mathfrak h) := \{\varepsilon_i \pm \varepsilon_j: 1\leq i<j\leq n\}$ and the lattice of $G$-integral weights is $P(G):=\bigoplus_{j=1}^n\Z\varepsilon_j \simeq \Z^n$. 

Given a dominant $G$-integral weight $\Lambda$, we will write $\pi_{\Lambda}$ for the irreducible representation of $G$ with highest weight $\Lambda$. 
Set $\Lambda_p= \varepsilon_1+\dots+\varepsilon_p$ for any $1\leq p\leq n$, $\Lambda_0=0$, and $\bar \Lambda_n = \Lambda_n-2\varepsilon_n$.
For $k\geq0$, define
\begin{equation}\label{eq1:pi_kp}
\pi_{k,p} =
\begin{cases}
0
&\quad\text{if $p=0$},\\
\pi_{k\varepsilon_1+\Lambda_p}
&\quad\text{if $1\leq p<n$},\\
\pi_{k\varepsilon_1+\Lambda_n}\oplus \pi_{k\varepsilon_1+\bar\Lambda_n}
&\quad\text{if $p=n$}.\\
\end{cases}
\end{equation}
It is well known that $V_{\pi_{0,p}}\simeq\bigwedge^p(\C^{2n})$ and $\operatorname{Sym}^{k+1}(\C^{2n}) \simeq V_{\pi_{k,1}}\oplus \operatorname{Sym}^{k-1}(\C^{2n})$ as $G$-modules.
For $\Gamma\subset G$, write $V_{\pi}^\Gamma$ for the subset of $\Gamma$-invariant elements in $V_{\pi}$.
For $k\geq 1$, we set
\begin{equation}\label{eq1:lambda_(k,p)}
\lambda_{k,p} =
\begin{cases}
0&\quad\text{if $p=-1$},\\
(k+p)(k+2n-2-p) &\quad\text{if $0\leq p\leq n-1$.}
\end{cases}
\end{equation}

The next theorem is well known (see for instance \cite[Thm.~4.2]{IkedaTaniguchi78}, \cite[Prop.~2.1]{Ikeda88}, \cite[Thm.~1.1]{LMR-repequiv}, \cite[Prop.~2.2]{LMR-onenorm}).
It describes the $p$-spectrum of $\Gamma\ba S^{2n-1}$ in algebraic terms.

\begin{theorem}\label{thm1:spectrum-general}
Fix $0\leq p\leq n-1$.
Each eigenvalue in $\spec_p(\Gamma\ba S^{2n-1})$ is of the form  $\lambda_{k,p-1}$ or $\lambda_{k,p}$ for some $k\geq 1$, with multiplicity
\begin{equation*}
\mult_{\Delta_{\Gamma,p}}(\lambda_{k,p-1}) = \dim V_{\pi_{k-1,p}}^\Gamma
\quad\text{ and }\quad
\mult_{\Delta_{\Gamma,p}}(\lambda_{k,p}) = \dim V_{\pi_{k-1,p+1}}^\Gamma.
\end{equation*}
\end{theorem}

A \emph{lens space} is a spherical space form $\Gamma\ba S^{2n-1}$ with $\Gamma$ cyclic, which can be assumed included in the standard maximal torus 
\begin{equation}\label{eq:maximaltorus}
T := \{\diag(R(\theta_1),\dots R(\theta_n)): \theta_1,\dots,\theta_n \in\R\} 
\qquad 
\left(R(\theta):= \left[\begin{smallmatrix}\cos(\theta)&-\sin(\theta) \\ \sin(\theta)&\cos(\theta) \end{smallmatrix}\right]\right)
\end{equation}
of $G$.
These spaces are parametrized as follows: for each $q\in\N$ and $s_1,\dots,s_n\in\Z$ satisfying $\gcd(q,s_1,\dots,s_n)=1$, set
$\gamma=\diag(R(2\pi s_1/q),\dots , R(2\pi s_1/q)) \in T$ and let $\Gamma$ be the cyclic group of order $q$ generated by $\gamma$;
the space $L(q;s_1,\dots,s_n) := \Gamma\ba S^{2n-1}$ is called an \emph{orbifold lens space}.
The group $\Gamma$  acts freely on $S^{2n-1}$ if and only if $\gcd(q,s_j)=1$ for all $j$, and in this case $L(q;s_1,\dots,s_n)$ is a \emph{lens space} (see \cite{Cohen-book} for details).
We will actually consider spaces more general than orbifold lens spaces.
Namely, good orbifolds of the form $\Gamma\ba S^{2n-1}$ with $\Gamma$ any discrete subgroup of $T$.
Note that $\Gamma$ must be finite and abelian.

For $\Gamma$ a finite subgroup of $T$, we want to determine $\spec_p(\Gamma\ba S^{2n-1})$ by using Theorem~\ref{thm1:spectrum-general}.
If $\Gamma\subset T$ is finite, we have (see \cite[Lem.~3.1]{LMR-onenorm} or \cite[Prop.~2.6]{Lauret-spec0cyclic}) that
\begin{equation}\label{eq1:dimV_pi_kp+1}
\dim V_{\pi_{k,p}}^\Gamma=  \sum_{\mu\in\mathcal L_\Gamma} m_{\pi_{k,p}}(\mu),
\end{equation}
where $\mathcal L_\Gamma = \{\mu\in P(G): \gamma^\mu=1\quad\forall\,\gamma\in\Gamma\}$, $\gamma^\mu = e^{\mu(H)}$ if $\gamma=\exp(H)$, and $m_{\pi_{k,p}}(\mu)$ denotes the multiplicity of the weight $\mu$ in the representation $\pi_{k,p}$. 
In the case of a lens space $L(q;s_1,\dots,s_n)=\Gamma\ba S^{2n-1}$, it turns out that
\begin{equation}\label{eq1:conglattice}
\mathcal L_\Gamma =
\left\{
a_1\varepsilon_1+\dots+a_n\varepsilon_n
\in P(G)\simeq \Z^n: a_1s_1+\dots+a_ns_n\equiv0\pmod q\right\}.
\end{equation}

The main tool will be a closed explicit expression for $m_{\pi_{k,p}}(\mu)$ obtained in \cite{LR-fundstring} (see Lemma~\ref{thm1:multip(k,p)} below).
The expression for $\mu =\sum_{j=1}^n a_j\varepsilon_j\in P(G)$ depends only on 
\begin{align}\label{eq:normaZ(mu)}
	\norma{\mu} := \sum_{j=1}^n |a_j| \quad \text{ and } \quad
	Z(\mu) := \#\{1\leq j\leq n: a_j=0\}.
\end{align}
Combining this expression, \eqref{eq1:dimV_pi_kp+1} and Theorem~\ref{thm1:spectrum-general}, we get an explicit formula for $\mult_{\Delta_{\Gamma,p}}(\lambda_{k,p})$ for $\Gamma$ a finite subgroup of $T$ (Theorem~\ref{thm:eigenvalue-multiplicities}).
The formula is written in terms of the following numbers:
for any subset $\mathcal L$ of $P(G)$, set
\begin{align}\label{eq1:N_L(k)N_L(k,l)}
N_{\mathcal L}(k,\ell) &= \#\{\mu\in \mathcal L: \norma{\mu}=k,\;Z(\mu)=\ell\},\\
N_{\mathcal L}(k)& = \#\{\mu\in \mathcal L: \norma{\mu}=k\}. 
\notag
\end{align}

Our next goal is to give an explicit description of $\spec_p(\Gamma\ba S^{2n-1})$ for $\Gamma\subset T$ finite, alternative to Theorem~\ref{thm:eigenvalue-multiplicities}.
This description uses generating functions, giving more elegant results.
Generating functions were first used for spectral problems concerning lens spaces by Ikeda and Yamamoto~\cite{IkedaYamamoto79}.

Ikeda in \cite{Ikeda88} associated to an arbitrary finite subgroup $\Gamma$ of $G=\SO(2n)$, the $n$ generating functions given by
\begin{equation}\label{eq1:F_Gamma^p}
F_\Gamma^{p}(z) =
\sum_{k\geq0}  \mult_{\Delta_{\Gamma,p}} (\lambda_{k+1,p}) z^{k}
=\sum_{k\geq0} \dim V_{\pi_{k,p+1}}^\Gamma z^k,
\end{equation}
for each $0\leq p\leq n-1$.
From Theorem~\ref{thm1:spectrum-general}, $\spec_p(\Gamma\ba S^{2n-1})$ is encoded by $F_\Gamma^{p-1}(z)$ and $F_\Gamma^{p}(z)$.
In particular (\cite[Prop.~2.1]{Ikeda88}),
\begin{equation}\label{eq1:p-isosp}
\Gamma\ba S^{2n-1}\text{ and }\Gamma'\ba S^{2n-1}
\text{ are \emph{$p$-isospectral}}
\text{ if and only if }
\begin{cases}
F_{\Gamma}^{p-1}(z)=F_{\Gamma'}^{p-1}(z),\\
F_{\Gamma}^{p}(z)=F_{\Gamma'}^{p}(z).
\end{cases}
\end{equation}
Ikeda gave expressions for the functions $F_\Gamma^p(z)$ for $0\leq p\leq n-1$ (see \cite[(2.13)]{Ikeda88}).
These expressions were used in the computational study \cite{GornetMcGowan06} of the $p$-spectrum of lens spaces (equation (3) in \cite{GornetMcGowan06} works only for $q$ prime).
As a consequence of Lemma~\ref{thm1:multip(k,p)}, we will give alternative expressions (only when $\Gamma\subset T$) in terms of the numbers in \eqref{eq1:N_L(k)N_L(k,l)}.

We next assume again that $\Gamma$ is any finite subgroup of $T$.
A formula for $m_{\pi_{k,1}}(\mu) = m_{\pi_{(k+1)\varepsilon_1}}(\mu)$ is previously known (see \cite[Lem.~3.2]{LMR-onenorm}). 
This fact and \eqref{eq1:dimV_pi_kp+1} applied to \eqref{eq1:F_Gamma^p} when $p=0$ give (see \cite[Thm.~3.6]{Lauret-spec0cyclic})
\begin{equation}\label{eq1:F_Gamma^0=rationalfcn}
F_\Gamma^{0}(z)
= \frac{1}{z}\left(\frac{\vartheta_{\mathcal L_\Gamma}(z)}{(1-z^2)^{n-1}} -1\right),
\end{equation}
where $\vartheta_{\mathcal L_\Gamma}(z)$ is the \emph{one-norm generating function} associated to $\mathcal L_\Gamma$ given by (see also \cite{Sole95}) 
\begin{equation}\label{eq1:theta_L}
\vartheta_{\mathcal L_\Gamma}(z)=\sum_{k\geq0} N_{\mathcal L_\Gamma}(k)z^k.
\end{equation}
Equation \eqref{eq1:F_Gamma^0=rationalfcn} gives a neat way to describe the $0$-spectrum of $\Gamma\ba S^{2n-1}$ for any $\Gamma\subset T$.

In the next result we exhibit an expression of $F_\Gamma^{p}(z)$, analogous to \eqref{eq1:F_Gamma^0=rationalfcn}, valid for any $p$.
For each $0\leq \ell\leq n$, set
\begin{equation}\label{eq1:theta_L^l}
\vartheta_{\mathcal L_\Gamma}^{(\ell)}(z) = \sum_{k\geq0} N_{\mathcal L_\Gamma}(k,\ell) z^k.
\end{equation}

\begin{theorem}\label{thm:F_Gamma^p}
Let $\Gamma$ be a finite subgroup of $T$.
For each $1\leq p\leq n$, there exist Laurent polynomials $A_{p}^{(\ell)}(z)$ for $0\leq\ell\leq n$, with $-p\leq \op{deg}(A_{p}^{(\ell)}(z))\leq n-\ell-1-2p$, satisfying that
\begin{equation}\label{eq1:F_Gamma^p-1-formula}
F_\Gamma^{p-1}(z) = \frac{1}{(1-z^2)^{n-1}}\sum_{\ell=0}^{n} \vartheta_{\mathcal L_\Gamma}^{(\ell)}(z) \; A_{p}^{(\ell)}(z)+\frac{(-1)^p}{z^p}
\end{equation}
for every $1\leq p\leq n$.
Explicit expressions for $A_{p}^{(\ell)}(z)$ are given in \eqref{eq3:A_pl}.
\end{theorem}

\smallskip

We next derive spectral consequences from Theorem~\ref{thm:F_Gamma^p}.
For $\Gamma$ and $\Gamma'$ finite subgroups of $T$, one obtains immediately from \eqref{eq1:F_Gamma^0=rationalfcn} the following geometric characterization proved in \cite[Thm.~3.6(i)]{LMR-onenorm},

\vspace{2mm}
\noindent\stepcounter{equation}(\theequation)\hspace*{0.05\textwidth}%
\parbox{0.7\textwidth}{$\Gamma\ba S^{2n-1}$ and $\Gamma'\ba S^{2n-1}$ are $0$-isospectral if and only if $\mathcal L_\Gamma$ and $\mathcal L_{\Gamma'}$ are \emph{$\norma{\cdot}$-isospectral} (i.e.\ $\vartheta_{\mathcal L_{\Gamma}}(z) = \vartheta_{\mathcal L_{\Gamma'}}(z)$).
}\\[2mm]
Generalizing this result, we obtain the following geometric characterization.

\begin{corollary}\label{cor:charact[0,p]}
Let $0\leq p_0\leq n-1$ and let $\Gamma$ and $\Gamma'$ be finite subgroups of $T$.
Then, $\Gamma\ba S^{2n-1}$ and $\Gamma'\ba S^{2n-1}$ are $p$-isospectral for all $0\leq p\leq p_0$ if and only if
\begin{equation}\label{eq1:condition[0,p_0]-isosp}
\sum_{\ell=0}^n \ell^h \, \vartheta_{\mathcal L_{\Gamma}}^{(\ell)}(z) =
\sum_{\ell=0}^n \ell^h \, \vartheta_{\mathcal L_{\Gamma'}}^{(\ell)}(z)
\qquad\text{for all }0\leq h\leq p_0,
\end{equation}
or equivalently,
\begin{equation}\label{eq1:condition[0,p_0]-isospN_kl}
\sum_{\ell=0}^n \ell^h \, N_{\mathcal L_{\Gamma}}(k,\ell) =
\sum_{\ell=0}^n \ell^h \, N_{\mathcal L_{\Gamma'}}(k,\ell)
\quad\text{for all } k\geq0,\qquad\text{for all }0\leq h\leq p_0.
\end{equation}
\end{corollary}

It turns out that the geometric characterization for lens spaces $p$-isospectral for all $p$, proved in \cite[Thm.~3.6(ii)]{LMR-onenorm}, coincides with the previous result when $p_0=n-1$ (see Remark~\ref{rem3:extreme-characterizations}). Namely, for $\Gamma$ and $\Gamma'$ finite subgroups of $T$,

\vspace{2mm}
\noindent\stepcounter{equation}(\theequation)\hspace*{0.05\textwidth}%
\parbox{0.7\textwidth}{$\Gamma\ba S^{2n-1}$ and $\Gamma'\ba S^{2n-1}$ are $p$-isospectral for all $p$ if and only if $\mathcal L_\Gamma$ and $\mathcal L_{\Gamma'}$ are \emph{$\norma{\cdot}^*$-isospectral} (i.e.\ $\vartheta_{\mathcal L_{\Gamma}}^{(\ell)}(z) = \vartheta_{\mathcal L_{\Gamma'}}^{(\ell)}(z)$ for all $\ell$).}\\[2mm]
Here, $\norma{\cdot}^*$-isospectrality means $N_{\mathcal L_\Gamma}(k,\ell)=N_{\mathcal L_{\Gamma'}}(k,\ell)$ for all $k\geq0$ and all $0\leq \ell\leq n$.
Examples of lens spaces $p$-isospectral for all $p$ can be found in \cite{LMR-onenorm}, \cite{LMR-survey} and \cite{DeFordDoyle14}.

Ikeda's expression \cite[(2.13)]{Ikeda88} for $F_\Gamma^p(z)$ ensures that $F_\Gamma^p(z)$ is a rational function.
As a last main result we obtain explicit expressions for those polynomials in terms of the associated lattice $\mathcal L_\Gamma$. 
This expression is very useful for explicit computational purposes. 
Moreover, it is required only finitely many calculations to determine the rational expression. 
Consequently, the $p$-spectrum of $\Gamma\ba S^{2n-1}$ is determined by a finite part of it.

With the above goal in mind, we will show that $\vartheta_{\mathcal L}^{(\ell)}(z)$ has a rational expression for any $\ell$, obtaining the required expression for $F_\Gamma^p(z)$ from Theorem~\ref{thm:F_Gamma^p}. 
We first introduce some more notation.
For $q$ a positive integer and $\mathcal L\subset P(G)$, we define
\begin{align}
C(q) &=\{\mu=\textstyle\sum_{i}a_i\varepsilon_i\in P(G): |a_i|<q\quad\forall\,i\},\label{eq1:reduceterms} \\
N_{\mathcal L}^{\textrm{red}}(k,\ell) &= \#\{\mu\in C(q)\cap\mathcal L: \norma{\mu}=k,\;Z(\mu)=\ell\},\notag \\
\Phi_{\mathcal L}^{(\ell)}(z) &= 
	\textstyle \sum_{k\geq0} N_{\mathcal L}^{\textrm{red}}(k,\ell) z^k.\notag
\end{align}
It is important to note that $\Phi_{\mathcal L}^{(\ell)}(z)$ is actually a polynomial of degree at most $(n-\ell)(q-1)$.
Indeed, if $\norma{\mu}>(n-\ell)(q-1)$, then $\mu\not\in C(q)$.
Hence, $N_{\mathcal L}^{\textrm{red}}(k,\ell)=0$ for all $k>(n-\ell)(q-1)$.
We set $q(\Gamma)=\min\{m\in\N: \gamma^m=1\quad\forall\,\gamma\in\Gamma\}$.

\begin{theorem}\label{thm:theta^l-rational}
Let $\Gamma$ be a finite subgroup of $T$ with $q(\Gamma)=q$.
Then, for each $0\leq \ell\leq n$,
\begin{equation*}
\vartheta_{\mathcal L_\Gamma}^{(\ell)} (z)
=\frac{1}{(1-z^{q})^{n-\ell}}\sum_{s=0}^{n-\ell} 2^s\binom{\ell+s}{s} z^{sq} \,\Phi_{\mathcal L_\Gamma}^{(\ell+s)}(z).
\end{equation*}
\end{theorem}

It was already shown in \cite[Thm.~3.9]{Lauret-spec0cyclic}, by using Ehrhart's theory for counting integral points in rational polytopes, that $\vartheta_{\mathcal L_\Gamma}(z)$ has a rational expression.
Actually, the author has recently found the article \cite{Sole95}, where a more general result was proved with the identical method.
The next corollary refines these results by giving an explicit expression for the polynomial at the numerator in terms of the polynomials $\Phi_{\mathcal L_\Gamma}^{(\ell)}(z)$.

\begin{corollary}\label{cor1:theta-rational}
Let $\Gamma$ be a finite subgroup of $T$ with $q(\Gamma)=q$.
Then,
\begin{equation*}
\vartheta_{\mathcal L_\Gamma}(z)
=\frac{1}{(1-z^{q})^{n}} \, \displaystyle\sum_{t=0}^{n} z^{tq}\sum_{\ell=t}^n \binom{\ell}{t}  \Phi_{\mathcal L_\Gamma}^{(\ell)}(z).
\end{equation*}
\end{corollary}

\section{Spectra on forms of lens spaces}\label{sec:proofs}
In this section we prove the statements in Section~\ref{sec:results}. 
We first describe in Theorem~\ref{thm:eigenvalue-multiplicities} the spectrum of the Hodge-Laplace operator $\Delta_{\Gamma,p}$ acting on $p$-forms of $\Gamma\ba S^{2n-1}$ with $\Gamma\subset T$.
Then, Theorem~\ref{thm:F_Gamma^p}, Corollary~\ref{cor:charact[0,p]}, Theorem~\ref{thm:theta^l-rational} and Corollary~\ref{cor1:theta-rational} are proved.
We first set up notation.

\begin{notation}\label{notation}
We recall that $G=\SO(2n)$, thus its Lie algebra and complexified Lie algebra are $\mathfrak g=\so(2n)$ and $\mathfrak g_\C=\so(2n,\C)$. 
We pick the Cartan subalgebra
\begin{equation*}
\mathfrak h:=\left\{ 
\diag\left(
\left[\begin{smallmatrix}0&- \theta_1\\  \theta_1&0\end{smallmatrix}\right]
, \dots,
\left[\begin{smallmatrix}0&- \theta_n\\  \theta_n&0\end{smallmatrix}\right]
\right)
:\theta_j\in\C\;\forall \, j
\right\}
\end{equation*}
of $\mathfrak g_\C$. 
For any $1\leq j\leq n$, let $\varepsilon_j\in \mathfrak h^*$ given by
$ 
\varepsilon_j \left(\diag\left(
\left[\begin{smallmatrix}0&- \theta_1\\  \theta_1&0\end{smallmatrix}\right]
, \dots,
\left[\begin{smallmatrix}0&- \theta_n\\  \theta_n&0\end{smallmatrix}\right]
\right)
\right)= i\theta_j.
$ 
Therefore, the root system associated to $(\mathfrak g_\C,\mathfrak h)$ (i.e.\ the root system of type $\mathrm{D}_n$) is given by
$
\Sigma(\mathfrak g_\C,\mathfrak h):= \{\pm\varepsilon_i\pm\varepsilon_j: 1\leq i<j\leq n\},
$
the lattice of $G$-integral weights is $P(G):=\bigoplus_{j=1}^n\Z\varepsilon_j$ and $\mu= \sum_{j=1}^n a_j\varepsilon_j\in P(G)$ is dominant if and only if $a_1\geq \dots\geq a_{n-1}\geq |a_n|$. 
\end{notation}

The following lemma, proved in \cite[Thm.~IV.1]{LR-fundstring}, gives a closed explicit expression for the multiplicity $m_{\pi_{k,p}}(\mu)$ of the weight $\mu$ in the representation $\pi_{k,p}$.
The formula depends only on the one-norm $\norma{\mu}$ of $\mu$ and the number of zero entries of $\mu$, denoted by $Z(\mu)$ (see \eqref{eq:normaZ(mu)}).
This fact was already shown in \cite[Lem.~3.3]{LMR-onenorm}.
We shall use the convention $\binom{b}{a}=0$ if $b<a$ or $a<0$.

\begin{lemma}\label{thm1:multip(k,p)}
	Let $k\geq0$, $1\leq p\leq n$ and let $\mu\in P(G)$.
	Write $r(\mu)=(k+p-\norma{\mu})/2$.
	If $r(\mu)$ is a non-negative integer, then
	\begin{align*}
		m_{\pi_{k,p}}(\mu)
		&= \sum_{j=1}^{p} (-1)^{j-1}  \sum_{t=0}^{\lfloor\frac{p-j}{2}\rfloor} \binom{n-p+j+2t}{t}  \sum_{\beta=0}^{p-j-2t} 2^{p-j-2t-\beta} \binom{n-Z(\mu)}{\beta} \binom{Z(\mu)}{p-j-2t-\beta}  \\
		&\quad \sum_{\alpha=0}^\beta \binom{\beta}{\alpha} \sum_{i=0}^{j-1} \binom{r(\mu)-i-p+\alpha+t+j+n-2}{n-2},
	\end{align*}
	and $m_{\pi_{k,p}}(\mu)=0$ otherwise.
\end{lemma}

For $\Gamma\subset T$ finite, $0\leq\ell\leq n$ and $k\geq1$, we set
\begin{multline}\label{eq3:M_Gamma(k,p)}
M_\Gamma(k,p)
= \sum_{\ell=0}^n \sum_{r=0}^{\lfloor\frac{k-1+p}{2}\rfloor} N_{\mathcal L_\Gamma}(k-1+p-2r,\ell)
\sum_{j=1}^{p} (-1)^{j-1}  \sum_{t=0}^{\lfloor\frac{p-j}{2}\rfloor} \binom{n-p+j+2t}{t}  \\
\sum_{\beta=0}^{p-j-2t} 2^{p-j-2t-\beta} \binom{n-\ell}{\beta} \binom{\ell}{p-j-2t-\beta}  \sum_{\alpha=0}^\beta \binom{\beta}{\alpha}
\sum_{i=0}^{j-1} \binom{r-i-p+j+\alpha+t+n-2}{n-2}.
\end{multline}
We recall from \eqref{eq1:lambda_(k,p)} that $\lambda_{k,p} =(k+p)(k+2n-2-p)$ if $0\leq p\leq n-1$, and $\lambda_{k,-1}=0$.

\begin{theorem}\label{thm:eigenvalue-multiplicities}
Let $\Gamma$ be a finite subgroup of $T$ and $0\leq p\leq n-1$.
Each eigenvalue in $\spec_p(\Gamma\ba S^{2n-1})$ is of the form $\lambda_{k,p-1}$ or $\lambda_{k,p}$ for some $k\geq1$, with multiplicities given by
\begin{equation*}
\mult_{\Delta_{\Gamma,p}} (\lambda_{k,p-1})= M_\Gamma(k,p-1)
\qquad\text{and}\qquad
\mult_{\Delta_{\Gamma,p}} (\lambda_{k,p})=M_\Gamma(k,p).
\end{equation*}
In particular, the $0$-spectrum of $\Gamma\ba S^{2n-1}$ has eigenvalues $k(k+2n-2)$ for any $k\geq0$, with
\begin{align*}
\mult_{\Delta_{\Gamma,0}} \big(k(k+2n-2)\big)
&= \sum_{r=0}^{\lfloor k/2\rfloor} N_{\mathcal L_\Gamma}(k-2r) \binom{r+n-2}{n-2}.
\end{align*}
\end{theorem}

\begin{proof}
From Theorem~\ref{thm1:spectrum-general}, it suffices to calculate $\dim V_{\pi_{k,p}}^\Gamma$ for $\Gamma\subset T$ finite and furthermore, \eqref{eq1:dimV_pi_kp+1} implies that $\dim V_{\pi_{k,p}}^\Gamma = \sum_{\mu\in\mathcal L_\Gamma} m_{\pi_{k,p}}(\mu)$.
Lemma~\ref{thm1:multip(k,p)} ensures that $m_{\pi_{k,p}}(\mu)=0$ if $k+p-\norma{\mu}\not\in2\Z_{\geq0}$, and moreover $m_{\pi_{k,p}}(\mu)=m_{\pi_{k,p}}(\mu')$ for $\mu$ and $\mu'$ satisfying $\norma{\mu}=\norma{\mu'}$ and $Z(\mu)=Z(\mu')$.
Thus,
\begin{align*}
\dim V_{\pi_{k,p}}^\Gamma
&= \sum_{\ell=0}^n  \sum_{r=0}^{\lfloor\frac{k+p}{2}\rfloor} \sum_{\mu\in\mathcal L_\Gamma:Z(\mu)=\ell\atop \norma{\mu}=k+p-2r} m_{\pi_{k,p}}(\mu)
= \sum_{\ell=0}^n  \sum_{r=0}^{\lfloor\frac{k+p}{2}\rfloor} N_{\mathcal L}(k+p-2r,\ell) \, m_{\pi_{k,p}}(\mu_0),
\end{align*}
where $\mu_0$ is any weight satisfying $\norma{\mu_0}=k+p-2r$ and $Z(\mu_0)=\ell$.
The rest follows by the formula for $m_{\pi_{k,p}}(\mu_0)$ in Lemma~\ref{thm1:multip(k,p)}.
\end{proof}

For $1\leq p\leq n-1$ and $0\leq \ell\leq n$, we set
\begin{align}\label{eq3:A_pl}
A_{p}^{(\ell)}(z) &= \sum_{j=1}^{p} (-1)^{j-1}  \sum_{t=0}^{\lfloor\frac{p-j}{2}\rfloor} \binom{n-p+j+2t}{t}
\\ &\quad
\sum_{\beta=0}^{p-j-2t} 2^{p-j-2t-\beta} \binom{n-\ell}{\beta} \binom{\ell}{p-j-2t-\beta}
\sum_{\alpha=0}^\beta \binom{\beta}{\alpha} \sum_{i=0}^{j-1} z^{p-2(j+t+\alpha-i)}. \notag
\end{align}

\begin{remark}\label{rem3:degreesA}
One can check that $A_{p}^{(\ell)}(z)$ is a Laurent polynomial of degree in between $-p$ and $p-2$.
Furthermore, $A_{p}^{(\ell)}(z)$ is a polynomial on the variable $\ell$, with coefficients in the space of complex Laurent polynomials on $z$, of degree at most $p$.
\end{remark}

\begin{proof}[Proof of Theorem~\ref{thm:F_Gamma^p}]
This proof is quite lengthy, so it will contain several claims in order to facilitate the reading.
Write $\mathcal L=\mathcal L_\Gamma$.
Set
\begin{align}
\Upsilon_{m}^{(\ell)}(z) &= \sum_{h=0}^{m} z^h \sum_{s=0}^{\lfloor\frac{h}{2}\rfloor} N_{\mathcal L}(h-2s,\ell)\binom{s+n-2}{n-2}, \label{eq3:Upsilon_m^l}\\
\label{eq3:B_Gammapl}
B_{\Gamma,p}^{(\ell)}(z) &= \sum_{j=1}^{p} (-1)^{j-1}  \sum_{t=0}^{\lfloor\frac{p-j}{2}\rfloor} \binom{n-p+j+2t}{t}
\sum_{\beta=0}^{p-j-2t} 2^{p-j-2t-\beta} \binom{n-\ell}{\beta}
\\ &\quad
\binom{\ell}{p-j-2t-\beta}
\sum_{\alpha=0}^\beta \binom{\beta}{\alpha} \sum_{i=0}^{j-1} z^{p-2(j+t+\alpha-i)} \;\Upsilon_{2(j+t+\alpha-i)-p-1}^{(\ell)}(z).
\notag
\end{align}

\begin{claim}\label{claim:F_Gamma^p-1}
$
\displaystyle F_\Gamma^{p-1}(z) = \sum_{\ell=0}^{n} \left(\frac{\vartheta_{\mathcal L}^{(\ell)}(z)}{(1-z^2)^{n-1}} \; A_{p}^{(\ell)}(z) -B_{\Gamma,p}^{(\ell)}(z)\right).
$
\end{claim}

\begin{proof}
\renewcommand{\qedsymbol}{$\blacksquare$}
Theorem~\ref{thm:eigenvalue-multiplicities} implies
$
F_\Gamma^{p-1}(z)= \sum_{k\geq0} M_\Gamma(k+1,p) z^k.
$
By \eqref{eq3:M_Gamma(k,p)} and reordering the sums, we obtain that
\begin{align}\label{eq3:F_Gamma^p-1-inicial}
F_\Gamma^{p-1}(z)
&= \sum_{\ell=0}^n
\sum_{j=1}^{p} (-1)^{j-1}  \sum_{t=0}^{\lfloor\frac{p-j}{2}\rfloor} \binom{n-p+j+2t}{t}  \\
&\quad
\sum_{\beta=0}^{p-j-2t} 2^{p-j-2t-\beta} \binom{n-\ell}{\beta} \binom{\ell}{p-j-2t-\beta}
\sum_{\alpha=0}^\beta \binom{\beta}{\alpha} \sum_{i=0}^{j-1}
G(z),\notag
\end{align}
where
$$
G(z) =  \sum_{k\geq0} z^k  \sum_{r=0}^{\lfloor\frac{k+p}{2}\rfloor} N_{\mathcal L}(k+p-2r,\ell) \binom{r-i-p+j+\alpha+t+n-2}{n-2}.
$$
Write $\gamma= i+p-j-\alpha-t$.
One can check that $\gamma\geq0$ for all allowed choices for $i$, $j$, $\alpha$ and $t$. 
We have that
\begin{align*}
G(z) 
&= \sum_{k\geq0} z^k  \sum_{r=\gamma }^{\lfloor\frac{k+p}{2}\rfloor} N_{\mathcal L}(k+p-2r,\ell) \binom{r-\gamma+n-2}{n-2} \\
&= \sum_{k\geq0} z^k  \sum_{s=0}^{\lfloor\frac{k+p-2\gamma}{2}\rfloor} N_{\mathcal L}(k+p-2\gamma-2s,\ell) \binom{s+n-2}{n-2} \\
&= z^{2\gamma-p} \sum_{h\geq p-2\gamma} z^{h}  \sum_{s=0}^{\lfloor\frac{h}{2}\rfloor} N_{\mathcal L}(h-2s,\ell) \binom{s+n-2}{n-2}.
\end{align*}
Here we made the changes of variables, $r=s+\gamma$ and $h=k+p-2\gamma$.
Hence
\begin{align*}
G(z)
&= z^{2\gamma-p} \left(\sum_{h\geq 0} z^h  \sum_{s=0}^{\lfloor\frac{h}{2}\rfloor} N_{\mathcal L}(h-2s,\ell) \binom{s+n-2}{n-2} - \Upsilon_{p-2\gamma-1}^{(\ell)}(z) \right) \\
&= z^{2\gamma-p}  \left(\frac{\vartheta_{\mathcal L}^{(\ell)}(z)}{(1-z^2)^{n-1}}
-\Upsilon_{p-2\gamma-1}^{(\ell)}(z) \right).
\end{align*}
This formula and \eqref{eq3:F_Gamma^p-1-inicial} imply the claim.
\end{proof}

By Claim~\ref{claim:F_Gamma^p-1}, it remains to show that $\sum_{\ell=0}^n B_{\Gamma,p}^{(\ell)}(z) = (-1)^{p+1}z^{-p}$.
Set
\begin{align}\label{eq3:C_pg^l}
C_{p,g}^{(\ell)} &=
\sum_{j=1}^{p} (-1)^{j-1}  \sum_{t=0}^{\lfloor\frac{p-j}{2}\rfloor} \binom{n-p+j+2t}{t}
\sum_{\beta=0}^{p-j-2t} 2^{p-j-2t-\beta} \binom{n-\ell}{\beta}
\binom{\ell}{p-j-2t-\beta}
\\ &\quad
\sum_{i=0}^{j-1} \binom{\beta}{p+i-j-t-g} . \notag
\end{align}

\begin{claim}
$\displaystyle B_{\Gamma,p}^{(\ell)}(z) = \sum_{g=0}^{p-1}  z^{2g-p} \;\Upsilon_{p-1-2g}^{(\ell)}(z) \;C_{p,g}^{(\ell)}$.
\end{claim}
\begin{proof}
\renewcommand{\qedsymbol}{$\blacksquare$}
We note that the sum over $\alpha$ in \eqref{eq3:B_Gammapl} can be extended to $0\leq \alpha  \leq p$ since $\binom{\beta}{\alpha}$ will vanish for $\alpha>\beta$.
We make the change of variables $g=p+i-j-t-\alpha$, thus $0\leq g\leq p-1$ and $\alpha=p+i-j-t-g $.
Hence
\begin{align*}
B_{\Gamma,p}^{(\ell)}(z)
&= \sum_{j=1}^{p} (-1)^{j-1}  \sum_{t=0}^{\lfloor\frac{p-j}{2}\rfloor} \binom{n-p+j+2t}{t}
\sum_{\beta=0}^{p-j-2t} 2^{p-j-2t-\beta} \binom{n-\ell}{\beta}
\\ &\quad
\binom{\ell}{p-j-2t-\beta}
\sum_{g=0}^{p-1} \sum_{i=0}^{j-1} \binom{\beta}{p+i-j-t-g}  z^{2g-p} \;\Upsilon_{p-1-2g}^{(\ell)}(z).
\end{align*}
The claim follows by reordering the sums.
\end{proof}

By definition, $\Upsilon_{m}^{(\ell)}(z)=0$ for $m<0$.
Moreover, $\Upsilon_{m}^{(\ell)}(z) =0$ if $n-\ell>m$, since $N_{\mathcal L}(k,\ell)=0$ for all $0\leq k<n-\ell$.
Hence,
\begin{equation}\label{eq3:sumB_pGamma^l}
\sum_{\ell=0}^n B_{\Gamma,p}^{(\ell)}(z)  = \sum_{g=0}^{\lfloor\frac{p-1}{2}\rfloor}  z^{2g-p} \sum_{\ell=n+1+2g-p}^{n}\Upsilon_{p-1-2g}^{(\ell)}(z) \;C_{p,g}^{(\ell)}.
\end{equation}

\begin{claim}\label{claim:C_pg^l}
For any $0\leq g\leq \tfrac{p-1}{2}$, $C_{p,g}^{(\ell)}=0$ if $n+1+2g-p\leq \ell<n$, and $C_{p,g}^{(n)}= (-1)^{p-1+g}\binom{n-1}{g}$.
\end{claim}

\begin{proof}
\renewcommand{\qedsymbol}{$\blacksquare$}
We first consider the case $\ell=n$.
One can easily see from \eqref{eq3:C_pg^l} that the sum over $\beta$ vanishes for $\beta>0$ since $n-\ell=0$.
Thus
\begin{align*}
C_{p,g}^{(n)}
&= \sum_{j=1}^{p} (-1)^{j-1}  \sum_{t=0}^{\lfloor\frac{p-j}{2}\rfloor} \binom{n-p+j+2t}{t}
2^{p-j-2t} \binom{n}{p-j-2t}
\sum_{i=0}^{j-1} \binom{0}{p+i-j-t-g} \\
&= \sum_{j=1}^{p} (-1)^{j-1}  \sum_{t=p-g-j}^{\lfloor\frac{p-j}{2}\rfloor} \binom{n-p+j+2t}{t}
2^{p-j-2t} \binom{n}{p-j-2t}
\end{align*}
since $\sum_{i=0}^{j-1} \binom{0}{p+i-j-t-g}$ is equal to $1$ if $p-g-j-t\leq 0\leq p-1-g-t$ and $0$ otherwise.
Note that the sum over $j$ in the last row is restricted to $p-g-j\leq \lfloor(p-j)/{2}\rfloor$, which is equivalent to $j\geq p-2g$.
By making the change of variables $k=p-j$, we obtain that
\begin{align}\label{eq3:C_pg^n}
C_{p,g}^{(n)}
&= \sum_{k=0}^{2g} (-1)^{p-k-1}  \sum_{t=k-g}^{\lfloor k/2\rfloor} \binom{n-k+2t}{t}
2^{k-2t} \binom{n}{k-2t}.
\end{align}

We now prove the assertion for $\ell=n$ by induction on $g$.
One can easily check that $C_{p,0}^{(n)}=(-1)^{p-1}$ from \eqref{eq3:C_pg^n}.
In the inductive step, we have that
\begin{multline*}
(-1)^{p-1}\, C_{p,g+1}^{(n)} =\sum_{k=0}^{2(g+1)} (-1)^{k}  \sum_{t=k-(g+1)}^{\lfloor\frac{k}{2}\rfloor} \binom{n-k+2t}{t}
2^{k-2t} \binom{n}{k-2t}
= (-1)^{p-1}\, C_{p,g}^{(n)} \\
-2n\binom{n-1}{g}
 +\binom{n}{g+1}
+ \sum_{k=0}^{2g} (-1)^{k}  \binom{n-2g-2+k}{k-g-1} 2^{2g+2-k} \binom{n}{2g+2-k}.
\end{multline*}
By making the change of variable $m=2g+2-k$ and the inductive hypothesis, we have that
\begin{align*}
C_{p,g+1}^{(n)}
&= C_{p,g}^{(n)}
+ (-1)^{p-1}\sum_{m=0}^{g+1} (-2)^{m}  \binom{n}{m} \binom{n-m}{g+1-m}
\\ &=
(-1)^{p-1}\sum_{u=0}^{g} (-1)^{u}\binom{n}{u}
+ (-1)^{p-1}\sum_{m=0}^{g+1} (-2)^{m}   \binom{g+1}{m} \binom{n}{g+1}
\\ &=
(-1)^{p-1}\sum_{u=0}^{g} (-1)^{u}\binom{n}{u}
+ (-1)^{p-1}(-1)^{g+1} \binom{n}{g+1}
=(-1)^{p-1} \sum_{u=0}^{g+1} (-1)^{u}\binom{n}{u}.
\end{align*}
The assertion follows by the well known identity
\begin{equation}\label{eq3:combinatorial_identity}
\sum_{u=0}^m(-1)^u \binom{a}{u}=(-1)^m\binom{a-1}{m}, \qquad a>0.
\end{equation}

We now consider the other cases.
We again use induction on $g$.
When $g=0$, we assume $n+1-p\leq \ell<n$ and  we have that
\begin{align*}
C_{p,0}^{(\ell)} &=
\sum_{j=1}^{p} (-1)^{j-1}  \sum_{t=0}^{\lfloor\frac{p-j}{2}\rfloor} \binom{n-p+j+2t}{t}
\sum_{\beta=0}^{p-j-2t} 2^{p-j-2t-\beta} \binom{n-\ell}{\beta}
\binom{\ell}{p-j-2t-\beta}
\\ &\quad
\sum_{i=0}^{j-1} \binom{\beta}{p+i-j-t}.
\end{align*}
Clearly, $\binom{\beta}{p+i-j-t}\neq0$ if and only if $0\leq p+i-j-t\leq \beta$, thus $i+t\leq 0$ since $\beta\leq p-j-2t$.
Thus, $i=t=0$ and
\begin{align*}
C_{p,0}^{(\ell)} &=
\sum_{j=1}^{p} (-1)^{j-1}
\sum_{\beta=0}^{p-j} 2^{p-j-\beta} \binom{n-\ell}{\beta}
\binom{\ell}{p-j-\beta}
\binom{\beta}{p-j} \\
&=
\sum_{j=1}^{p} (-1)^{j-1}
\binom{n-\ell}{p-j}
=(-1)^{p-1}\sum_{j=0}^{p-1} (-1)^{j} \binom{n-\ell}{j}
=\binom{n-\ell-1}{p-1} =0
\end{align*}
since $n-\ell-1<p-1$ by assumption.
The third equality in the last row follows by \eqref{eq3:combinatorial_identity}.

We now assume that $C_{p,g}^{(\ell)}=0$, $g+1\leq \tfrac{p-1}{2}$ and $n+1+2(g+1)-p\leq \ell$ ($\iff n-\ell\leq p-2g-3$).
In order to use induction, we first note that
\begin{equation*}
\sum_{i=0}^{j-1} \binom{\beta}{p+i-j-t-g-1}
= \binom{\beta}{p-j-t-g-1} +\sum_{i=0}^{j-1} \binom{\beta}{p+i-j-t-g} -\binom{\beta}{p-1-t-g},
\end{equation*}
which is the last part in the expression of $C_{p,g+1}^{(\ell)}$ in \eqref{eq3:C_pg^l}.
Hence,
\begin{align*}
C_{p,g+1}^{(\ell)}
&=\sum_{j=1}^{p} (-1)^{j-1}  \sum_{t=0}^{\lfloor\frac{p-j}{2}\rfloor} \binom{n-p+j+2t}{t}
\sum_{\beta=0}^{p-j-2t} 2^{p-j-2t-\beta} \binom{n-\ell}{\beta}
\binom{\ell}{p-j-2t-\beta}
\\ &\quad
\binom{\beta}{p-j-t-g-1}
+C_{p,g}^{(\ell)}
-\sum_{j=1}^{p} (-1)^{j-1}  \sum_{t=0}^{\lfloor\frac{p-j}{2}\rfloor} \binom{n-p+j+2t}{t}
\\ &\quad
\sum_{\beta=0}^{p-j-2t} 2^{p-j-2t-\beta} \binom{n-\ell}{\beta} \binom{\ell}{p-j-2t-\beta}
\binom{\beta}{p-1-t-g}.
\end{align*}
We have that $C_{p,g}^{(\ell)}=0$ by induction.
Moreover, the last term vanishes since $\binom{n-\ell}{\beta}\binom{\beta}{p-1-t-g}=0$ in the corresponding cases.
Indeed, if this number were nonzero, then $p-1-t-g\leq \beta\leq n-\ell\leq p-2g-3$, thus $g+2\leq t$, which is a contradiction since $p-1-t-g\leq \beta\leq p-j-2t$ yields $t\leq g+1-j\leq g$.

By making the change of variables $\gamma=p-j-2t-\beta$ to the remaining term, we obtain that
\begin{align*}
C_{p,g+1}^{(\ell)} &=
\sum_{j=1}^{p} (-1)^{j-1}  \sum_{t=0}^{\lfloor\frac{p-j}{2}\rfloor} \binom{n-p+j+2t}{t}
\sum_{\gamma=0}^{p-j-2t} 2^{\gamma} \binom{\ell}{\gamma}\binom{n-\ell}{p-j-2t-\gamma}
\binom{p-j-2t-\gamma}{p-j-t-g-1}.
\end{align*}
Clearly, $\binom{p-j-2t-\gamma}{p-j-t-g-1}\neq0$ if and only if $0\leq p-j-t-g-1\leq p-j-2t-\gamma$, thus $t+\gamma\leq g+1$.
This implies that the sum over $t$ goes from $0$ to $\min(g+1,\lfloor\frac{p-j}{2}\rfloor)$.
Since $g+1\leq \lfloor\frac{p-j}{2}\rfloor$ if and only if $j\leq p-2g-2$, we have that
\begin{align*}
C_{p,g+1}^{(\ell)}
&= \sum_{j=1}^{p-2g-2} (-1)^{j-1}  \sum_{t=0}^{g+1} \binom{n-p+j+2t}{t}
\sum_{\gamma=0}^{g+1-t} 2^{\gamma} \binom{\ell}{\gamma}\binom{n-\ell}{p-j-2t-\gamma}
\binom{p-j-2t-\gamma}{g+1-t-\gamma} \\
&
+ \sum_{j=p-2g-1}^{p} (-1)^{j-1}  \sum_{t=0}^{\lfloor\frac{p-j}{2}\rfloor} \binom{n-p+j+2t}{t}
\sum_{\gamma=0}^{g+1-t} 2^{\gamma} \binom{\ell}{\gamma}\binom{n-\ell}{p-j-2t-\gamma}
\binom{p-j-2t-\gamma}{g+1-t-\gamma} \\
&=   \sum_{t=0}^{g+1}
\sum_{\gamma=0}^{g+1-t} 2^{\gamma} \binom{\ell}{\gamma}
\sum_{j=1}^{p-2g-2} (-1)^{j-1} \binom{n-p+j+2t}{t} \binom{n-\ell}{p-j-2t-\gamma}
\binom{p-j-2t-\gamma}{g+1-t-\gamma} \\
&
+ \sum_{t=0}^{g} \sum_{\gamma=0}^{g+1-t} 2^{\gamma} \binom{\ell}{\gamma}\sum_{j=p-2g-1}^{p-2t} (-1)^{j-1}   \binom{n-p+j+2t}{t}
\binom{n-\ell}{p-j-2t-\gamma}
\binom{p-j-2t-\gamma}{g+1-t-\gamma}.
\end{align*}
The last equality follows by reordering the sums in both terms, with the sum extremes handled with particular care in the second row.

In the first term in the formula above, the sum over $j$ goes actually from $2g+3-2t-\gamma$ since $\binom{n-\ell}{p-j-2t-\gamma}\neq 0$ implies that $p-j-2t-\gamma\leq n-\ell\leq p-2g-3$.
By making the change of variables $k=p-j-2t-\gamma$ in both terms we obtain that
\begin{align*}
C_{p,g+1}^{(\ell)}
&=   \sum_{t=0}^{g+1}
\sum_{\gamma=0}^{g+1-t} 2^{\gamma} \binom{\ell}{\gamma}
\sum_{k=2g+2-2t-\gamma}^{p-2g-3} (-1)^{p-k-\gamma-1} \binom{n-k-\gamma}{t} \binom{n-\ell}{k}
\binom{k}{g+1-t-\gamma} \\
&
+ \sum_{t=0}^{g} \sum_{\gamma=0}^{g+1-t} 2^{\gamma} \binom{\ell}{\gamma}\sum_{k=-\gamma}^{2g+1-2t-\gamma}  (-1)^{p-k-\gamma-1} \binom{n-k-\gamma}{t} \binom{n-\ell}{k}
\binom{k}{g+1-t-\gamma}.
\end{align*}
Note that in the second term, the sum over $k$ goes from $0$ in order to $\binom{n-\ell}{k}\neq0$, and the sum over $t$ can be extended up to $g+1$ since it vanishes for $t=g+1$.
Hence,
\begin{align*}
C_{p,g+1}^{(\ell)}
&=   (-1)^{p-1}\sum_{t=0}^{g+1}
\sum_{\gamma=0}^{g+1-t} (-2)^{\gamma} \binom{\ell}{\gamma}
\sum_{k=0}^{p-2g-3} (-1)^{k} \binom{n-k-\gamma}{t} \binom{n-\ell}{k}
\binom{k}{g+1-t-\gamma}  \\
&=   (-1)^{p-1}
\sum_{k=0}^{p-2g-3} (-1)^{k}\binom{n-\ell}{k}
\sum_{\gamma=0}^{g+1} (-2)^{\gamma} \binom{\ell}{\gamma}
\sum_{t=0}^{g+1-\gamma} \binom{n-k-\gamma}{t} \binom{k}{g+1-t-\gamma}.
\end{align*}

The well known identity $\sum_{t=0}^m\binom{a}{t}\binom{b}{m-t}=\binom{a+b}{m}$ implies that the sum over $t$ in the formula above equals $\binom{n-k-\gamma+k}{g+1-\gamma}= \binom{n-\gamma}{g+1-\gamma}$, which does not depend on $k$.
Hence
\begin{align*}
C_{p,g+1}^{(\ell)}
&=   (-1)^{p-1}
\sum_{\gamma=0}^{g+1} (-2)^{\gamma} \binom{\ell}{\gamma}
\binom{n-\gamma}{g+1-\gamma}
\sum_{k=0}^{p-2g-3} (-1)^{k}\binom{n-\ell}{k}\\
&=   (-1)^{p-1}
\sum_{\gamma=0}^{g+1} (-2)^{\gamma} \binom{\ell}{\gamma}
\binom{n-\gamma}{g+1-\gamma}
(-1)^{p-2g-3}\binom{n-\ell-1}{p-2g-3}.
\end{align*}
The last row follows by \eqref{eq3:combinatorial_identity}.
We conclude that $C_{p,g+1}^{(\ell)}=0$ since $n-\ell-1<n-\ell\leq p-2g-3$ by assumption.
\end{proof}

\begin{claim}
$\displaystyle \sum_{\ell=0}^n B_{\Gamma,p}^{(\ell)}(z) =\frac{(-1)^{p-1}}{z^p}. $
\end{claim}
\begin{proof}
\renewcommand{\qedsymbol}{$\blacksquare$}
Equation \eqref{eq3:sumB_pGamma^l} and Claim~\ref{claim:C_pg^l} imply that
\begin{align*}
\sum_{\ell=0}^n B_{\Gamma,p}^{(\ell)}(z)
&= \sum_{g=0}^{\lfloor\frac{p-1}{2}\rfloor} z^{2g-p} \; \Upsilon_{p-1-2g}^{(n)}(z) \;C_{p,g}^{(n)}.
\end{align*}
By Claim~\ref{claim:C_pg^l}, since $N_{\mathcal L}(k,n)=1$ if $k=0$ and $0$ otherwise, we have that
\begin{align*}
\sum_{\ell=0}^n B_{\Gamma,p}^{(\ell)}(z)
&= \frac{(-1)^{p-1}}{z^p}\sum_{g=0}^{\lfloor\frac{p-1}{2}\rfloor} z^{2g} (-1)^{g}\binom{n-1}{g}\; \sum_{h=0}^{\lfloor\frac{p-1}{2}\rfloor-g}  z^{2h}\binom{h+n-2}{n-2}\\
&= \frac{(-1)^{p-1}}{z^p} \sum_{m=0}^{\lfloor\frac{p-1}{2}\rfloor} z^{2m}
\sum_{g=0}^{m} (-1)^{g}\binom{n-1}{g} \binom{m-g+n-2}{n-2}
.
\end{align*}
We note that $\sum_{g=0}^{m} (-1)^{g}\binom{n-1}{g} \binom{m-g+n-2}{n-2}$ is the $m$-th term of the series
\begin{multline*}
\left(\sum_{h\geq0}\binom{h+n-2}{n-2}z^{2h}\right)
\left(\sum_{g=0}^{\lfloor\frac{p-1}{2}\rfloor} (-z^2)^{g}\binom{n-1}{g} \right) \\
=\frac{1}{(1-z^2)^{n-1}} \left( (1-z^2)^{n-1}-\sum_{g=\lfloor\frac{p-1}{2}\rfloor+1}^{n-1} (-z^2)^{g}\binom{n-1}{g}\right) \\
=1-\frac{1}{(1-z^2)^{n-1}} \sum_{g=\lfloor\frac{p-1}{2}\rfloor+1}^{n-1} (-z^2)^{g}\binom{n-1}{g},
\end{multline*}
which is $1$ for $m=0$ and $0$ for $1\leq m\leq \lfloor\frac{p-1}{2}\rfloor$.
Then $\sum_{\ell=0}^n B_{\Gamma,p}^{(\ell)}(z)=(-1)^{p-1}/z^p$ as claimed.
\end{proof}
This completes the proof of Theorem~\ref{thm:F_Gamma^p}.
\end{proof}

Before proving Corollary~\ref{cor:charact[0,p]}, we make an observation.

\begin{remark}\label{rem3:extreme-characterizations}
Corollary~\ref{cor:charact[0,p]} contains the characterizations in \cite[Thm.~3.6]{LMR-onenorm} in the extreme cases $p_0=0$ and $p_0=n-1$.
Indeed, if $p_0=0$, then \eqref{eq1:condition[0,p_0]-isosp} tells us that $ \vartheta_{\mathcal L_{\Gamma}}(z)= \sum_{\ell=0}^n \vartheta_{\mathcal L_{\Gamma}}^{(\ell)}(z) = \sum_{\ell=0}^n \vartheta_{\mathcal L_{\Gamma'}}^{(\ell)}(z) = \vartheta_{\mathcal L_{\Gamma'}}(z)$, or equivalently that, $\mathcal L_\Gamma$ and $\mathcal L_{\Gamma'}$ are $\norma{\cdot}$-isospectral.

When $p_0=n-1$, the condition \eqref{eq1:condition[0,p_0]-isosp} consists of $n$ equations in the $n$ variables $\vartheta_{\mathcal L_{\Gamma}}^{(\ell)}(z) -\vartheta_{\mathcal L_{\Gamma'}}^{(\ell)}(z)$, $0\leq \ell\leq n-1$ (clearly $\vartheta_{\mathcal L_{\Gamma}}^{(n)}(z) =\vartheta_{\mathcal L_{\Gamma'}}^{(n)}(z) =1$).
The matrix associated to this linear equation is the Vandermonde matrix, which is non-singular.
Hence, \eqref{eq1:condition[0,p_0]-isosp} is equivalent to have $\vartheta_{\mathcal L_{\Gamma}}^{(\ell)}(z) =\vartheta_{\mathcal L_{\Gamma'}}^{(\ell)}(z)$ for all $0\leq \ell\leq n$, that is, $\mathcal L_\Gamma$ and $\mathcal L_{\Gamma'}$ are $\norma{\cdot}^*$-isospectral.
\end{remark}

\begin{proof}[Proof of Corollary~\ref{cor:charact[0,p]}]
Write $\mathcal L=\mathcal L_\Gamma$ and $\mathcal L'=\mathcal L_{\Gamma'}$.
By Theorem~\ref{thm:F_Gamma^p}, we have that
\begin{equation}\label{eq3:F_Gamma^p-F_Gamma'^p}
F_{\Gamma}^{p}(z) - F_{\Gamma'}^{p}(z) = \frac{1}{(1-z^2)^{n-1}}\sum_{\ell=0}^{n} \left(\vartheta_{\mathcal L}^{(\ell)}(z) - \vartheta_{\mathcal L'}^{(\ell)}(z)\right)A_{p+1,\ell}(z)
\end{equation}
for all $0\leq p\leq n-1$.

On the one hand, we know that $\Gamma\ba S^{2n-1}$ and $\Gamma'\ba S^{2n-1}$ are $p$-isospectral for all $0\leq p\leq p_0$ if and only if $F_{\Gamma}^{p}(z) = F_{\Gamma'}^{p}(z)$ for every $0\leq p\leq p_0$ by \eqref{eq1:p-isosp}.
On the other hand, \eqref{eq1:condition[0,p_0]-isosp} is equivalent to
$
\sum_{\ell=0}^{n} \left(\vartheta_{\mathcal L}^{(\ell)}(z) - \vartheta_{\mathcal L'}^{(\ell)}(z)\right)A_{p+1,\ell}(z)=0,
$
since $A_{p+1,\ell}(z)$ is a polynomial on $\ell$ of degree $\leq p_0+1$ (see Remark~\ref{rem3:degreesA}) for every $0\leq p\leq p_0$.
These both facts complete the proof.
\end{proof}

\begin{proof}[Proof of Theorem~\ref{thm:theta^l-rational}]
In \cite[Thm.~4.2]{LMR-onenorm}, the following formula was proved, which actually holds for any sublattice $\mathcal L$ of $P(G)\simeq \Z^n$ satisfying that $\mu\in\mathcal L$ if and only if $\mu+q\nu\in\mathcal L$ for any $\nu\in P(G)$: for $a\geq0$ and $0\leq r<q$,
\begin{equation}\label{eq3:N-N^red}
N_{\mathcal L}(aq+r,\ell )= \sum_{s=0}^{n-\ell } 2^s\binom{\ell +s}{s} \sum_{t=s}^{a}
\binom{t-s+n-\ell -1}{n-\ell -1} \,
N_{\mathcal L}^{\mathrm{red}}((a-t)q+r,\ell +s).
\end{equation}

One can check that $\mathcal L_\Gamma$ with $\Gamma\subset T$ finite and $q(\Gamma)=q$, satisfies that $\mu\in\mathcal L_\Gamma$ if and only if $\mu+q\nu\in\mathcal L_\Gamma$ for any $\nu\in P(G)$.
Write $\mathcal L=\mathcal L_\Gamma$.
Applying \eqref{eq3:N-N^red} to \eqref{eq1:theta_L^l}, we obtain that
\begin{align*}
\vartheta_{\mathcal L}^{(\ell)}(z)
&= \sum_{r=0}^{q-1} \sum_{a\geq0} N_{\mathcal L_\Gamma}(aq+r,\ell) z^{aq+r} \\
&=\sum_{s=0}^{n-\ell } 2^s\binom{\ell +s}{s}
\sum_{r=0}^{q-1} \sum_{a\geq0}  z^{aq+r} \sum_{t=s}^{a}
\binom{t-s+n-\ell -1}{n-\ell -1} \,
N_{\mathcal L}^{\mathrm{red}}((a-t)q+r,\ell +s) \\
&=\sum_{s=0}^{n-\ell } 2^s\binom{\ell +s}{s} z^{sq}
\sum_{r=0}^{q-1} z^{r} \sum_{b\geq0}  z^{bq} \sum_{u=0}^{b}
\binom{u+n-\ell -1}{n-\ell -1} \,
N_{\mathcal L}^{\mathrm{red}}((b-u)q+r,\ell +s).
\end{align*}
In the last row we made the change of variables $t=u+s$ and $a=b+s$.

Since $\sum_{k\geq0} \binom{k+n-\ell-1}{n-\ell-1} z^{kq}=(1-z^q)^{-(n-\ell)}$, we get
\begin{align*}
\vartheta_{\mathcal L}^{(\ell)}(z)
&= \sum_{s=0}^{n-\ell } 2^s\binom{\ell +s}{s} z^{sq} \sum_{r=0}^{q-1} z^r
\left( \frac{1}{(1-z^q)^{n-\ell}} \sum_{h\geq0} N_{\mathcal L}^{\mathrm{red}}(hq+r,\ell +s) z^{hq}\right) \\
&= \frac{1}{(1-z^q)^{n-\ell}} \sum_{s=0}^{n-\ell } 2^s\binom{\ell +s}{s} z^{sq} \sum_{k\geq0} N_{\mathcal L}^{\mathrm{red}}(k,\ell +s) z^{k},
\end{align*}
which completes the proof.
\end{proof}

\begin{proof}[Proof of Corollary~\ref{cor1:theta-rational}]
Write $\mathcal L=\mathcal L_\Gamma$.
Theorem~\ref{thm:theta^l-rational} implies that
\begin{align*}
\vartheta_{\mathcal L}(z)
    & = \sum_{\ell=0}^n \vartheta_{\mathcal L}^{(\ell)}(z)
      = \frac{1}{(1-z^q)^n}\sum_{\ell=0}^n (1-z^q)^\ell \sum_{s=0}^{n-\ell} 2^s\binom{\ell+s}{s} z^{sq} \,\Phi_{\mathcal L}^{(\ell+s)}(z) \\
    & = \frac{1}{(1-z^q)^n}\sum_{\ell=0}^n \left(\sum_{r=0}^\ell \binom{\ell}{r} (-1)^rz^{rq}\right) \left(\sum_{s=0}^{n-\ell} 2^s\binom{\ell+s}{s} z^{sq} \,\Phi_{\mathcal L}^{(\ell+s)}(z)\right).
\end{align*}
The two sums over the square $\{(r,s):0\leq r\leq \ell,\, 0\leq s\leq n-\ell\}$ can be replaced by the two sums over $\{(t,s): 0\leq t\leq n,\, 0\leq s\leq n\}$ by making the change of variables $t=r+s$, since $\binom{\ell}{r}=0$ for any $r>\ell$ and by letting $\Phi_{\mathcal L}^{(\ell)}(z)=0$ for any $\ell>n$.
Hence
\begin{align*}
\vartheta_{\mathcal L}(z)
    & = \frac{1}{(1-z^q)^n}\sum_{\ell=0}^n \sum_{t=0}^n z^{tq}  \sum_{s=0}^{n} \Phi_{\mathcal L}^{(\ell+s)}(z)\, 2^s\binom{\ell+s}{s}  \binom{\ell}{t-s} (-1)^{t-s} \\
    & =  \sum_{t=0}^n  \frac{z^{tq}}{(1-z^q)^n} \sum_{s=0}^{n} 2^s (-1)^{t-s} \sum_{\ell=s}^{n} \Phi_{\mathcal L}^{(\ell)}(z)\binom{\ell}{s}  \binom{\ell-s}{t-s}.
\end{align*}
We next make a new change of variables.
The sums over the quadrilateral $\{(s,\ell): 0\leq s\leq n,\, s\leq \ell\leq n\}$ is replaced by the sums over the triangle $\{(s,\ell): 0\leq \ell\leq n,\, \ell\leq s\leq n\}$ since $\binom{\ell}{s}=0$ if $s>n\leq\ell$.
Then
$
\vartheta_{\mathcal L}(z)
     =  \sum_{t=0}^n  \frac{z^{tq}}{(1-z^q)^n} \sum_{\ell=0}^{n} \binom{\ell}{t} \Phi_{\mathcal L}^{(\ell)}(z) \sum_{s=0}^{n} \binom{t}{s}2^s (-1)^{t-s} ,
$
which concludes the proof since the last sum is equal to $(-1+2)^n=1$ and $\binom{\ell}{t}=0$ for $\ell<t$.
\end{proof}

\bibliographystyle{plain}

\begin{thebibliography}{LMR15b}

\bibitem[BL17]{BoldtLauret-onenormDirac}
	{\sc S. Boldt, E.A. Lauret}.
	{\it An explicit formula for the Dirac multiplicities on lens spaces.}
	J. Geom. Anal. \textbf{27} (2017), 689--725.
	DOI: \href{http://dx.doi.org/10.1007/s12220-016-9695-x} {10.1007/s12220-016-9695-x}.

\bibitem[Co]{Cohen-book}
	{\sc M. Cohen}.
	A course in simple-homotopy theory.
	{\it Grad. Texts in Math.} \textbf{10}.
	Springer-Verlag, New York-Heidelberg-Berlin, 1970.
	
\bibitem[DD14]{DeFordDoyle14}
    {\sc D. DeFord, P. Doyle}.
    {\it Cyclic groups with the same Hodge series}.
    Rev. Un. Mat. Argentina \textbf{59}:2 (2018), 241--254. 
	
\bibitem[DG89]{DeTurckGordon89}
	{\sc D. DeTurck, C. Gordon}.
	{\it Isospectral deformations II: Trace formulas, metrics, and potentials}.
	Comm Pure Appl. Math. \textbf{42}:8 (1989), 1067--1095.
	DOI: \href{http://dx.doi.org/10.1002/cpa.3160420803} {10.1002/cpa.3160420803}.

\bibitem[Go86]{Gordon86}
	{\sc C. Gordon}.
	{\it Riemannian manifolds isospectral on functions but not on 1-forms}.
	J. Differential Geom. \textbf{24}:1 (1986), 79--96.
	
\bibitem[Gt00]{Gornet00}
	{\sc R. Gornet}.
	{\it Continuous families of {R}iemannian manifolds, isospectral on functions but not on 1-forms}.
	J. Geom. Anal. \textbf{10}:2 (2000), 281--298.
	DOI: \href{http://dx.doi.org/10.1007/BF02921826} {10.1007/BF02921826}.
		
\bibitem[GM06]{GornetMcGowan06}
	{\sc R. Gornet, J. McGowan}.
	{\it Lens Spaces, isospectral on forms but not on functions}.
	LMS J. Comput. Math. \textbf{9} (2006), 270--286.
	DOI: \href{http://dx.doi.org/10.1112/S1461157000001273} {10.1112/S1461157000001273}.
	
\bibitem[Ik80]{Ikeda80_isosp-lens}
	{\sc A. Ikeda}.
	{\it On lens spaces which are isospectral but not isometric}.
	Ann. Sci. \'Ecole Norm. Sup. (4) \textbf{13}:3 (1980), 303--315.
	
\bibitem[Ik88]{Ikeda88}
	{\sc A. Ikeda}.
	{\it Riemannian manifolds $p$-isospectral but not $p+1$-isospectral}.
	In \textsl{Geometry of manifolds ({M}atsumoto, 1988)}, 383--417,
	Perspect. Math. \textbf{8}, 1989.
		
\bibitem[IT78]{IkedaTaniguchi78}
	{\sc A. Ikeda, Y. Taniguchi}.
	{\it Spectra and eigenforms of the Laplacian on $S^n$ and $P^n(\mathbb C)$}.
	Osaka J. Math. \textbf{15}:3 (1978), 515--546.
	
\bibitem[IY79]{IkedaYamamoto79}
	{\sc A. Ikeda, Y. Yamamoto}.
	{\it On the spectra of 3-dimensional lens spaces}.
	Osaka J. Math. \textbf{16}:2 (1979), 447--469.
	
\bibitem[La16]{Lauret-spec0cyclic}
	{\sc E.A. Lauret}.
	{\it Spectra of orbifolds with cyclic fundamental groups}.
	Ann. Global Anal. Geom. \textbf{50}:1 (2016), 1--28.
	DOI: \href{http://dx.doi.org/10.1007/s10455-016-9498-0} {10.1007/s10455-016-9498-0}.

\bibitem[La17]{Lauret-computationalstudy}
	{\sc E.A. Lauret}.
	{\it A computational study on lens spaces isospectral on forms}.
	\href{https://arxiv.org/abs/1703.03077}{arXiv:1703.03077} (2017).
	
\bibitem[LMR15]{LMR-repequiv}
	{\sc E.A. Lauret, R.J. Miatello, J.P. Rossetti}.
	{\it Representation equivalence and p-spectrum of constant curvature space forms}.
	J. Geom. Anal. \textbf{25}:1 (2015), 564--591.
	DOI: \href{http://dx.doi.org/10.1007/s12220-013-9439-0} {10.1007/s12220-013-9439-0}.
	
\bibitem[LMR16a]{LMR-onenorm}
	{\sc E.A. Lauret, R.J. Miatello, J.P. Rossetti}.
	{\it Spectra of lens spaces from 1-norm spectra of congruence lattices.}
	Int. Math. Res. Not. IMRN \textbf{2016}:4 (2016), 1054--1089.
	DOI: \href{http://dx.doi.org/10.1093/imrn/rnv159} {10.1093/imrn/rnv159}.
	
\bibitem[LMR16b]{LMR-survey}
	{\sc E.A. Lauret, R.J. Miatello, J.P. Rossetti}.
	{\it Non-strongly isospectral spherical space forms.}
	In \textsl{Mathematical Congress of the Americas},
	Contemp. Math. \textbf{656}, Amer. Math. Soc., Providence, RI, 2016.
	DOI: \href{http://dx.doi.org/10.1090/conm/656/13104} {10.1090/conm/656/13104}.

\bibitem[LR17]{LR-fundstring}
	{\sc E.A. Lauret, F. Rossi Bertone}.
	{\it Multiplicity formulas for fundamental strings of representations of classical Lie algebras}.
	J. Math. Phys. 58 (2017), 111703. 
	DOI: \href{http://dx.doi.org/10.1063/1.4993851} {10.1063/1.4993851}.
	
\bibitem[MR01]{MR-p-iso}
	{\sc R.J. Miatello, J.P. Rossetti}.
	{\it Flat manifolds isospectral on $p$-forms.}
	Jour. Geom. Anal. \textbf{11}:4 (2001), 649--667.
	DOI: \href{http://dx.doi.org/10.1007/BF02930761} {10.1007/BF02930761}.
	
\bibitem[Mi64]{Milnor64}
	{\sc J. Milnor}.
	{\it Eigenvalues of the Laplace operator on certain manifolds.}
	Proc. Natl. Acad. Sci. USA \textbf{51}:4 (1964), 542.
	
\bibitem[MH16]{MohadesHonori16}
	{\sc H. Mohades, B. Honari}.
	{\it Harmonic-counting measures and spectral theory of lens spaces}.
	C. R. Math. Acad. Sci. Paris \textbf{354}:12 (2016), 1145--1150.
	DOI: \href{http://dx.doi.org/10.1016/j.crma.2016.10.016} {10.1016/j.crma.2016.10.016}.
	
\bibitem[MH17]{MohadesHonori17}
	{\sc H. Mohades, B. Honari}.
	{\it On a relation between spectral theory of lens spaces and Ehrhart theory}.
	Indag. Math. \textbf{28}:2 (2017), 556--565.
	DOI: \href{http://dx.doi.org/10.1016/j.indag.2017.01.003} {10.1016/j.indag.2017.01.003}.	

\bibitem[Sh11]{Shams11}
    {\sc N. Shams Ul Bari}.
    {\it Orbifold lens spaces that are isospectral but not isometric.}
    Osaka J. Math \textbf{48}:1 (2011), 1--40.
	
\bibitem[S95]{Sole95}
	{\sc P. Sol\'e}.
	{\it Counting lattice points in pyramids.}
	Discrete Math. \textbf{139} (1995), 381--392.
	DOI: \href{http://dx.doi.org/10.1016/0012-365X(94)00142-6} {10.1016/0012-365X(94)00142-6}.

\bibitem[Su85]{Sunada85}
	{\sc T. Sunada}.
	{\it Riemannian coverings and isospectral manifolds.}
	Ann. of Math. (2) \textbf{121}:1 (1985), 169--186.
	DOI: \href{http://dx.doi.org/10.2307/1971195}{10.2307/1971195}.
	
\bibitem[Vi80]{Vigneras80}
	{\sc M. Vign{\'e}ras}.
	{\it Vari\'et\'es riemanniennes isospectrales et non isom\'etriques.}
	Ann. of Math. (2) \textbf{112}:1 (1980), 21--32.
	DOI: \href{http://dx.doi.org/10.2307/1971319} {10.2307/1971319}.
\end{thebibliography}

\end{document}